\documentclass[12pt]{amsart}

\usepackage{latexsym, amssymb}

\newcommand{\be}{\begin{equation}}
\newcommand{\ee}{\end{equation}}
\newcommand{\beq}{\begin{eqnarray}}
\newcommand{\eeq}{\end{eqnarray}}

\newtheorem{prop}{Proposition}[section]
\newtheorem{theo}[prop]{Theorem}
\newtheorem{lemm}[prop]{Lemma}

\newtheorem{rema}[prop]{Remark}

\newtheorem{defi}[prop]{Definition}

\def\begeq{\begin{equation}}
\def\endeq{\end{equation}}

\begin{document}

\title{A Note on the Entropy of Mean Curvature Flow}

\begin{abstract}
The entropy of a hypersurface is given by the supremum over all F-functionals with varying centers and scales, and is invariant under rigid motions and dilations. As a consequence of Huisken's monotonicity formula, entropy is non-increasing under mean curvature flow. We show here that a compact mean convex hypersurface with some low entropy is diffeomorphic to a round sphere. We will also prove that a smooth self-shrinker with low entropy is exact a hyperplane.
\end{abstract}

\renewcommand{\subjclassname}{\textup{2010} Mathematics Subject Classification}
\author{Chao Bao}

\address{Chao Bao, Key Laboratory of Pure and Applied mathematics, School of Mathematics Science, Peking University,
Beijing, 100871, P.R. China.} \email{chbao@126.com}

\date{2014}
\maketitle

\markboth{Chao Bao}{}

 \maketitle

 \section{Introduction}

 The F-functional of a hypersurface $\Gamma \subset \textbf{R}^{n+1}$ is defined as
 $$F(\Gamma) = (4 \pi)^{-n/2} \int_{\Gamma} e^{-\frac{|x|^2}{4}}$$
 whereas the entropy of $\Gamma$ is given by
 \begin{equation}
\lambda(\Gamma) = \sup_{x_0 \in \textbf{R}^{n+1}, t_0 >0} (4 \pi t_0)^{-n/2} \int_{\Gamma} e^{-\frac{|x-x_0|^2}{4t_0}}
 \end{equation}
 If taking a transformation of the integral, we can also get
 \begin{equation}
   \lambda(\Gamma) = \sup_{x_0 \in \textbf{R}^{n+1}, t_0 >0} (4 \pi)^{-n/2} \int_{t_0\Gamma + x_0} e^{-\frac{|x|^2}{4}}
 \end{equation}
 By section 7 in \cite{CM}, the entropy of a self-shrinker is equal to the value of the F-functional $F$ and thus no supremum is needed. In \cite{S}, Stone computed the entropy for generalized cylinders $\textbf{S}^k \times \textbf{R}^{n-k}$. He showed that $\lambda(\textbf{S}^n)$ is decreasing in $n$ and
 $$\lambda(\textbf{S}^1) = \sqrt{\frac{2\pi}{e}} \approx 1.5203 > \lambda(\textbf{S}^2) = \frac{4}{e} \approx 1.4715 > \lambda(\textbf{S}^3) > \cdots > 1 = \lambda(\textbf{R}^{n})$$
 Moreover, a simple computation shows that $\lambda(\Sigma \times \textbf{R}) = \lambda(\Sigma)$.

 Mean curvature flow is a parameter family of hypersurfaces $\{M_t\} \subset \textbf{R}^{n+1}$ which evolves under the following equation:
 \begin{equation}
   (\partial_t X(p,t))^{\perp} = - H(p,t) \nu(p,t)
 \end{equation}
 Here $\overrightarrow{H} = - H \nu$ is the mean curvature vector of $M_t$, $H = div_{M_t} \nu$, $\nu$ is the outward unit normal, $X$ is the position vector and $\cdot^{\perp}$ denotes the projection on the normal space.

 We denote $\Phi(x,t) = (4 \pi t)^{-n/2} e^{-\frac{|x|^2}{4t}}$ and $\Phi_{(y,\tau)} = \Phi(x-y, \tau - t)$, Huisken's monotonicity implies that for any $(y,\tau) \in \textbf{R}^{n+1} \times \textbf{R}$, $t_1$ and $t_2$ with $t_2 < t_1 < \tau$ we have
 \begin{equation}
   \int_{M_{t_1}} \Phi_{(y,\tau)} \leq \int_{M_{t_2}} \Phi_{(y,\tau)}
 \end{equation}
 As a consequence of Huisken's monotonicity formula, entropy is non-increasing under mean curvature flow.

 A hypersurface $\Gamma \subset \textbf{R}^{n+1}$ is a self-shrinker if it satisfies
\begin{equation}\label{selfshrinker}
  H = \frac{\langle X,\nu \rangle}{2}
\end{equation}
It can be proved that, if $\Gamma$ is a self-shrinker, then $\Gamma_t = \sqrt{-t} \Gamma$ satisfies the mean curvature flow equation, see lemma 2.2 in \cite{CM}.

A non-compact hypersurface $\Sigma \subset \textbf{R}^{n+1}$ is said to be with polynomial volume growth if there are constants $C$ and $d$ so that for all $r \geq 1$
\begin{equation}
  Vol(B_r(0) \cap \Sigma) \leq Cr^d.
\end{equation}
where $B_r(0)$ denote the ball centered at origin $0$ with radius $r$ in $\textbf{R}^{n+1}$.

In \cite{H}, Huisken showed that mean curvature flow stating at any smooth compact convex initial hypersurface in $\textbf{R}^{n+1}$ remains convex and smooth until it becomes extinct at a point and if we rescale the flow about the point in space-time where it becomes extinct, then the rescalings converge to round spheres. In \cite{HS}, Huisken and Sinestrari developed a theory for mean curvature flow with surgery for two-convex hypersurfaces in $\textbf{R}^{n+1} (n \geq 3)$, and classified all of the closed two-convex hypersurfaces. In \cite{CM}, Colding and Minicozzi found a piece-wise mean curvature flow, under which they could prove that assuming a uniform diameter bound the piece-wise mean curvature flow starting from any closed surface in $\textbf{R}^3$ will become extinct in a round point.

Inspired by \cite{CIMW}, we expect to study hypersurfaces  perspective from entropy, i.e. whether we can classify all of the mean convex hypersurfaces under some entropy condition. Specially, when the entropy of a closed mean convex hypersurface is no more than $\lambda(\textbf{S}^{n-2})$, whether we can classify all of this kind of hypersurfaces like the result of Huisken and Sinestrari, see \cite{HS}. As a first step to our goal, in this note, we will prove that under some entropy condition a mean convex closed hypersurface is diffeomorphic to a round sphere.

It seems to the author that, entropy plays  similar roles as energy does in harmonic map theory. For example, in harmonic map theory one has $\epsilon$-regularity theorem \cite{L} \cite{SW}, Liouville type theorem for harmonic maps with small energy \cite{EL}, and uniqueness of harmonic maps with small energy \cite{mS} etc. If comparing self-shrinkers as harmonic maps one has similar results on the entropies of self-shrinkers. So it also motivates the author to do this work.

\begin{theo}\label{main2}
  Suppose $M_0 \subset \textbf{R}^{n+1}$ is a smooth closed embedded hypersurface with mean curvature $H > 0$. If $\lambda(M_0) \leq \min\{\lambda(\textbf{S}^{n-1}), \frac{3}{2}\}$, then it is diffeomorphic to a round sphere $\textbf{S}^n$.
\end{theo}
Moreover, we can get the following Bernstein type theorem for self-shrinkers under some low entropy condition.
\begin{theo}\label{main1}
  Suppose $\Gamma$ is a smooth non-compact embedded self-shrinker with polynomial volume growth, there exists a constant $\epsilon >0$, such that if $\lambda(\Gamma)< 1+ \epsilon$ then $\Gamma$ must be a hyperplane.
\end{theo}

It should be pointed out that, under the assumption $\lambda(M_0) < 2$, it is easy to check that all tangent flows must be multiplicity one, and as a sequel, we will not need to mention this in the proof of main theorems. In the proof of main theorems, we will use similar techniques from \cite{E}, \cite{CIMW} etc.

{\bf Acknowledgements} The author is very grateful to Professor Yuguang Shi for discussing this result and many helpful comments on this problem.

\section{Tangent flows of mean curvature flows}Throughout this paper, unless otherwise mentioned, we will always assume $M_0$ is a smooth closed embedded hypersurface in $\textbf{R}^{n+1}$, and $\{M_t\}$ is a mean curvature flow starting from $M_0$.

Let $(x_0, t_0) \in \textbf{R}^{n+1} \times \textbf{R}$ be a fixed point in the space-time, and $\lambda > 0$ be a positive constant in $\textbf{R}$. We say that  $\{M_s^{\lambda}\}$ is a parabolic rescaling of $\{M_t\}$ at $(x_0 , t_0)$ if it satisfies
\begin{equation}
  M^{\lambda}_s = \lambda^{-1}(M_{\lambda^{2}s + t_0} - x_0)
\end{equation}
 where $s \in (-\lambda^{-2} t_0, 0)$. It is easy to check that $\{M_s^{\lambda}\}$ also satisfies mean curvature flow equation. For any hypersurface $M$ in $\textbf{R}^{n+1}$, we say $x_0$ is a regular point of $M$, if there is an open neighbourhood $U_0 \subset \textbf{R}^{n+1}$ of $x_0$, such that $M$ is smooth in $M \cap U_0$. Moreover, we say $M$ is regular, if every point of $M$ is a regular point.

 \begin{defi}
   We say that a parameter of hypersurfaces $\{\Gamma_s\}_{s<0}$ is a tangent flow of $\{M_t\}$, if there exists a sequence of positive numbers $\{\lambda_j\}$, $\lambda_j \rightarrow$ 0 as $j \rightarrow \infty$, such that $M^{\lambda_j}_s \hookrightarrow \Gamma_s$ as Randon measures for each $s<0$.
 \end{defi}
 We will denote $M^j_s = M^{\lambda_j}_s$ for simplicity without confusion. About the existence of tangent flows, we have the following lemma:
 \begin{lemm}[see \cite{I2}] \label{tang.exist}
   Suppose $\{M_t\}$ is a mean curvature flow, and $M_0$ is a smooth embedded hypersurface, then for any time-space point $(x_0,t_0) \in \textbf{R}^{n+1} \times \textbf{R}$ there is a parameter of hypersurfaces $\{\Gamma_s\}_{s<0}$ and a sequence of positive numbers $\{\lambda_j\}$, $\lambda_j \rightarrow$ 0 as $j \rightarrow \infty$, such that $M^{j}_s \hookrightarrow \Gamma_s$ as Radon measures for each $s<0$.
 \end{lemm}

 Moreover, by Lemma 8 of \cite{I1}, we know that $\Gamma_s = \sqrt{-s}\Gamma_{-1}$, and $\Gamma_{-1}$ is a weak solution of self-shrinker equation (\ref{selfshrinker}). Furthermore by Huisken's monotonicity formula, we can prove the following point-wise convergence lemma:

 \begin{lemm} \label{tan.conv}
   If $\{\Gamma_s\}_{s<0}$ is a tangent flow of $\{M_t\}$ at $(x_0,t_0)$, and $\{M^j_s\}$ is the corresponding sequence of parabolic transformation of $\{M_t\}$, then $\{M^j_s\}$ converge to $\{\Gamma_s\}$ as Hausdorff distance for each $s<0$.
 \end{lemm}
 \begin{proof}
   Because $M_0$ is closed and embedded, we can prove that for any fixed $t$, $T < t < t_0$ for some $T > 0$, there is a constant $V = V(Vol(M_0),T)$ such that $Vol(B_{r}(0) \cap M_t) \leq Vr^n$ for all $r >0$, and all $T \leq t < t_0$, see Lemma 2.9 in \cite{CM}. Furthermore, it is easy to check that
   \begin{equation}\label{dis1}
   Vol(B_{r}(0) \cap M_s^j) \leq Vr^n
   \end{equation}
   for all $r >0$ and all $\lambda_j^{-2}(T-t_0) \leq s < 0$.

   Since $\{M_s^j\}$ is a also a mean curvature flow, then by Huisken's monotonicity formula for any $x_0 \in \textbf{R}^{n+1}$, and any $s_2 < s_1 < s_0$, we have
   \begin{equation}\label{dis2}
   \int_{M_{s_1}^j} \Phi_{(x_0,s_0)} \leq \int_{M_{s_2}^j} \Phi_{(x_0,s_0)}
   \end{equation}
   By (\ref{dis1}), smoothness of the function $\Phi_{(x_0,s_0)}$, and the measure convergence of $\{M^j_s\}$, as $j \rightarrow \infty$ for every $s < s_0$ we have
   \begin{equation} \label{conv}
    \int_{M^j_s} \Phi_{(x_0,s_0)}  \rightarrow \int_{\Gamma_s} \Phi_{(x_0,s_0)}
   \end{equation}
   Combining this with (\ref{dis2}), we have for any $s_2 < s_1 < s_0$,
   \begin{equation}
   \int_{\Gamma_{s_1}} \Phi_{(x_0,s_0)} \leq \int_{\Gamma_{s_2}} \Phi_{(x_0,s_0)}
   \end{equation}
   so
   $$\lim_{s \nearrow s_0} \int_{\Gamma_{s}} \Phi_{(x_0,s_0)} $$ exists.

   Suppose there are a sequence $\{x_j\}$, $x_j \in M^j_{s_0}$ and a point $y \in \textbf{R}^{n+1}$ satisfying $\lim_{j \rightarrow \infty} x_j= y$. It is easy to see that if we prove $y \in \Gamma_{s_0}$, then we get the lemma.

   For any smooth embedded mean curvature flow $\{\widehat{M}_t\}$, so it is easy to check that if $\widehat{x} \in \widehat{M}_{s_0}$
   \begin{equation}\label{dis3}
     \lim_{s \rightarrow s_0} \int_{\widehat{M}_t} \Phi_{(\widehat{x}, s_0)} = 1
   \end{equation}
   Moreover, it is also easy to check that for any $\widehat{x} \notin \widehat{M}_{s_0}$, we have
   \begin{equation}\label{dis4}
     \lim_{s \rightarrow s_0} \int_{\widehat{M}_{s_0}} \Phi_{(\widehat{x}, s_0)} = 0
   \end{equation}
   That is to say, if
   $$\lim_{s \rightarrow s_0} \int_{\widehat{M}_{s}} \Phi_{(\widehat{x}, s_0)} \neq 0$$
   we must have $\widehat{x} \in \widehat{M}_{s_0}$.Actually, we do not need to assume $\{\widehat{M}_t\}$ is smooth and embedded here.

   Furthermore, if $\{\widehat{M}_t\}$ is only smooth in a neighbourhood of $\widehat{x}$, we also have (\ref{dis3}) and (\ref{dis4}), see  P.66 in \cite{E}.

   For the sequence $\{x_j\}$ and $y$, it is also easy to prove the following result like (\ref{conv}) under the the condition of (\ref{dis1}), smoothness of the function $\Phi_{(x_0,s_0)}$, and the measure convergence of $\{M^j_s\}$:
   \begin{equation}\label{conv2}
     \int_{M^j_s} \Phi_{(x_j,s_0)}  \rightarrow \int_{\Gamma_s} \Phi_{(y,s_0)}
   \end{equation}
   as $j \rightarrow \infty$. By Huisken's monotonicity formula and (\ref{dis3}), from (\ref{conv2}) we get that
   \begin{equation}
     \int_{\Gamma_s} \Phi_{(y,s_0)} \geq 1
   \end{equation}
   Then we take $s \rightarrow s_0$, we have
   $$\lim_{s \rightarrow s_0} \int_{\Gamma_s} \Phi_{(\widehat{x}, s_0)} \geq 1$$
   Thus we must have $y \in \Gamma_{s_0}$ and complete the proof.
 \end{proof}
 In the following subsections, see lemma {\ref{mainlemma}}, we will further prove that if a tangent flow is smooth and embedded, we even have smooth convergence.

\section{Partial regularity for mean curvature flows}
We will need a partial regularity theorem due to Ecker, see theorem 5.6 in \cite{E}.

Before stating Ecker's theorem, we need to introduce a test function, which plays an important role in Ecker's local monotonicity, see theorem 4.17 in \cite{E}. Define
$$\phi_{\rho}(x,t) = (1 - \frac{|x|^2 + 2nt}{\rho^2})^3_{+}$$
and its translates
$$\phi_{(x_0 , t_0), \rho}(x,t) = \phi_{\rho}(x - x_0 , t - t_0)$$
For an open subset $U$ of $\textbf{R}^{n+1}$, there is a radius $\rho_0 > 0$ such that
$$B_{\sqrt{1+2n} \rho_0}(x_0) \times (t_0 - \rho_0^2, t_0) \subset U \times (t_1 ,t_0).$$
For all $\rho \in (0, \rho_0)$ and $t \in (t_0 - \rho_0^2, t_0)$ then we have
$$spt \phi_{(x_0, t_0), \rho} \subset B_{\sqrt{\rho^2 - 2n(t -t_0)}}(x_0) \subset B_{\sqrt{1+2n} \rho_0}(x_0) \subset U$$
The Gaussian density at $(x_0 , t_0)$ of mean curvature flow $\{M_t\}$ is defined as
\begin{equation}\label{gaussden}
  \Theta(M_t ,x_0, t_0) = \lim_{t \nearrow t_0} \int_{M_t} \Phi_{(x_0, t_0)}
\end{equation}
It is easy to check that, if $x_0$ is a regular point of $M_{t_0}$ then $ \Theta(M_t ,x_0, t_0) = 1$.

\begin{theo}[Ecker's local monotonicity, \cite{E}]\label{locmono}
  Let $\{M_t\}_{t \in (t_1 ,t_0)}$ be a smooth,properly embedded solution of mean curvature flow in an open set $U \subset \textbf{R}^{n+1}$. Then for every $x_0 \in U$ there is a $\rho_0 \in (0, \sqrt{t_0 - t_1})$ such that for all $\rho  \in (0, \rho_0]$ and $t \in (t_0 - \rho^2 , t_0)$ we have
  $$spt\phi_{(x_0 , t_0), \rho}(\cdot , t) \subset U$$
  and
  $$\frac{d}{dt} \int_{M_t} \Phi_{(x_0 , t_0)} \phi_{(x_0 ,t_0), \rho} \leq - \int_{M_t} |\overrightarrow{H}(x) - \frac{(x - x_0)^{\perp}}{2(t - t_0)}|^2 \Phi_{(x_0 , t_0)} \phi_{(x_0 ,t_0), \rho}$$
  Since the right-hand side is non-positive and
  $$\phi_{(x_0 , t_0) , \rho}(x_0 ,t_0) = 1$$
  for every $\rho \in (0, \rho_0]$, this implies that the locally defined Gaussian density
  $$\Theta(M_t, x_0 , t_0) \equiv \lim_{t\nearrow t_0} \int_{M_t} \Phi_{(x_0, t_0)}\phi_{(x_0, t_0), \rho}$$
  exists, is independent of $\rho$ and for global solutions agrees with the Gaussian density defined in (\ref{gaussden}). Furthermore, for every $t \in (t_0 - \rho^2, t_0),$
  $$\Theta(M_t, x_0 , t_0) \leq \int_{M_t}\Phi_{(x_0, t_0)}\phi_{(x_0, t_0), \rho} $$
\end{theo}
The following partial regularity theorem is due to by B.White \cite{W}, and in \cite{E} Ecker proves a similar result using the local monotonicity formula, here we present Ecker's version of B.White's partial regularity theorem.
\begin{theo}[Ecker, \cite{E}]\label{Ecker}
  Suppose $\{M_t\}$ is a smooth, properly embedded solution of mean curvature flow in $U \times (t_1 ,t_0)$ which reaches $x_0$ at time $t_0$, and $U$ is an open set in $\textbf{R}^{n+1}$. Then there exist constants $\epsilon_0 >0$ and $c_0 >0$ such that whenever
  $$\Theta(M_t , x_0 ,  t_0) \leq 1+ \epsilon_0$$
  holds at $x_0 \in U$, then
  $$|A(x)|^2 \leq \frac{c_0}{\rho^2}$$
  for some $\rho > 0$ and for all $x \in M_t \cap B_{\rho}(x_0)$ and $t \in (t_0 - \rho^2, t_0)$. In particular, $x_0$ is a regular point at time $t_0$.
\end{theo}

In Ecker's proof, he actually proved the following result:
\begin{theo}
    Whenever $\{M_t\}$ is a smooth, properly embedded solution of mean curvature flow in $U \times (t_1, t_0)$ which reaches $x_0$ at time $t_0$, and $B_{\rho}(x_0) \times (t_0 - 2\rho^2 , t_0) \subset U \times (t_1 , t_0)$, if there exists constants $\epsilon_0 >0$ and $c_0 > 0$ such that if
    $$\int_{M_t} \Phi_{(y,\tau)} \phi_{(y,\tau),\rho_0} \leq \epsilon_0$$
    for all $(y,\tau) \in B_{\rho}(x_0) \times (t_0 - \rho^2 , t_0)$ and $t \in (\tau - \rho^2, \tau)$, and $\rho_0$ is chosen to make sure that $spt \phi_{(y,\tau),\rho_0} \subset U$, then we have
    $$|A(x)|^2 \leq \frac{c_0}{\rho^2}$$
    for all $x \in M_t \cap B_{\rho}(x_0)$ and $t \in (t_0 - \rho^2, t_0)$.
  \end{theo}

  \begin{rema}
    In the original version of Ecker's theorem, he didn't point out what exactly the constant $c_0$ depends on. However, throughout his proof, the author think $c_0$ depends on $\epsilon_0$, $U$ and $t_0 - t_1$. Whatever, we can still prove Theorem \ref{main1} following his proof.
  \end{rema}

   \subsection{Proof of theorem \ref{main1}} Following the same technique given by Ecker in proving Theorem \ref{Ecker}, now we prove Theorem \ref{main1}.
   \begin{lemm}\label{thm1lemm}
    Let $M_t$ be an smooth complete embedded ancient solution of mean curvature flow which exists in  $(-\infty, 0]$. Assuming the origin $0 \in M_0$, then there exists a constant $\epsilon >0$, such that for any such ancient solution $M_t$ and any $ R >0$, if for all $(y, \tau) \in B_R(0) \times (-\infty, 0]$, $M_t$ satisfies
    $$\int_{M_t} \Phi_{(y,\tau)} < 1+\epsilon,$$
    then we have
    $$ (\sigma R)^2 \sup_{(-(1-\sigma)^2 R^2 , 0)} \sup_{M_t \cap B_{(1-\sigma)R}(0)} |A|^2 \leq C_0,$$
    for all $\sigma \in (0,1)$, and $C_0$ does not depend on $R$ and $M_t$.
  \end{lemm}
  \begin{proof}
    Suppose the lemma is not correct. Then for every $j \in \emph{N}$ one can find a smooth, complete embedded solution $\{M^j_t\}$ which reaches $0 \in \textbf{R}^{n+1}$ at time 0 and some $R_j >0$ such that for all $(y, \tau) \in B_{R_j}(0) \times (\infty , 0]$,
    $$\int_{M^j_t} \Phi_{(y,\tau)} \leq 1+ \frac{1}{j}$$
    holds but
    $$\gamma^2_j \equiv \sup_{\sigma \in (0,1)}((\sigma R_j)^2 \sup_{(-(1-\sigma)^2 R^2_j,0)} \sup_{M^j_t \cap B_{(1-\sigma)R_j}} |A|^2) \rightarrow \infty$$
    as $j \rightarrow \infty$. In particular, one can find a $\sigma_j \in (0,1)$ for which
    $$\gamma^2_j = (\sigma_j R_j)^2 \sup_{(-(1-\sigma_j)^2 R^2_j,0)} \sup_{M^j_t \cap B_{(1-\sigma_j)R_j}} |A|^2$$
    and a point
    $$y_j \in M^j_{\tau_j} \cap \overline{B}_{(1-\sigma_j)R_j}$$
    at a time $[-(1-\sigma_j)^2R_j^2, 0]$ so that
    $$\gamma_j^2 = \sigma_j^2 R_j^2 |A(y_j)|^2.$$
    If we choose $\sigma = \frac{1}{2} \sigma_j$, we have
    $$\sigma_j^2 R^2_j \sup_{(-(1-\frac{\sigma_j}{2})^2R_j^2, 0)} \sup_{M^j_t \cap B_{(1-\sigma_j/2)R_j}(0)} |A|^2 \leq 4\gamma_j^2$$
    that is
    $$\sup_{(-(1-\frac{\sigma_j}{2})^2R_j^2, 0)} \sup_{M^j_t \cap B_{(1-\sigma_j/2)R_j}(0)} |A|^2 \leq 4|A(y_j)|^2.$$
    Since $(\tau_j - \frac{\sigma_j^2}{4}R^2_j , \tau_j) \subset (-(1-\frac{\sigma_j}{2})^2R_j^2, 0)$ and $B_{\sigma_jR_j/2}(y_j) \subset B_{(1-\sigma_j/2)R_j}(0)$
    so we can get
    $$\sup_{(\tau_j - \sigma_j^2R^2_j/4 , \tau_j)} \sup_{M^j_t \cap B_{\sigma_jR_j/2(y_j)}} |A|^2 \leq 4|A(y_j)|^2. $$
    Now let
    $$\lambda_j = |A(y_j)|^{-1}$$
    and define
    $$\widetilde{M}^j_s = \frac{1}{\lambda_j} (M^j_{\lambda^2_j s + \tau_j} - y_j)$$
    for $s \in [\lambda^{-2}_j \sigma_j^2 R^2_j/4, 0]$.

    Then $\{\widetilde{M}^j_s\}$ is a smooth solution of mean curvature flow satisfying
    $$0 \in \widetilde{M}^j_0, |A(0)| = 1$$
    and
    $$\sup_{(\lambda^{-2}_j \sigma^2_j R^2_j/4, 0)} \sup_{\widetilde{M}^j_s \cap B_{\lambda^{-1}_j \sigma_j R_j/2}(0)} \leq 4$$
    for every $j \in \emph{N}$. Since
    $$\lambda_j^{-2} \sigma^2_j R^2_j = \gamma^2_j \rightarrow \infty$$
    we have for every $R>0$ and sufficiently large $j$ depending on $R$,
    $$\sup_{(-R^2 , 0)}\sup_{\widetilde{M}^j_s \cap B_{R}(0)} |A|^2 \leq 4.$$
    By curvature estimates for mean curvature flow we know that for every $j$, $\{\widetilde{M}^j_s\}$ is smooth and have uniform curvature estimate on any compact subset of time-space $\emph{R}^{n+1} \times \emph{R}$. These allow us to apply Arzela-Ascoli theorem to conclude that a subsequence of $\{\widetilde{M}^j_s\}$ converges smoothly on compact subsets of $\emph{R}^{n+1} \times \emph{R}$ to a smooth solution $\{M'_s\}_{s \leq 0}$ of mean curvature flow. Moreover, for $\{M'_s\}_{s \leq 0}$ we can get
    $$0 \in M'_0, |A(0)| = 1$$
    and
    $$|A(y)| \leq 4$$
    for $y \in M'_s, s \leq 0$.
    By our assumption, we know that
    $$\int_{\widetilde{M}^j_s} \Phi \leq 1+ \frac{1}{j}$$
    for all $s \in (-\lambda^{-2}_j \sigma^2_j R^2_j/4, 0)$.
    By the decreasing of $\int_{M^j_t} \Phi_{(y_j,\tau_j)}$ and the smoothness of $\widetilde{M}^j_0$, we also get
    $$\int_{\widetilde{M}^j_s} \Phi \geq 1.$$
    Now we take the limit for $j \rightarrow \infty$, we have
    $$\int_{M'_s} \Phi = 1$$
    for all $s<0$.

    At last, following the same argument in the proof of Theorem 5.6 in \cite{E}, and we complete the proof.
   \end{proof}

   \textit{Proof of Theorem \ref{main1}}: Under the condition of Theorem \ref{main1}, it is easy to check that $\Gamma_{t} = \sqrt{-t+1}\Gamma$ is a self-shrinking ancient solution of mean curvature flow. By Huisken's monotonicity formula and the definition of entropy, we see that $\Gamma_t$ satisfies all the condition needed in Lemma \ref{thm1lemm}, so we get the estimate for $\Gamma$
   $$ (\sigma R)^2 \sup _{(-(1-\sigma)^2 R^2 , 0)} \sup_{M_t \cap B_{(1-\sigma)R}(0)} \leq C_0$$
   for all $\sigma \in (0,1)$. If we take $\sigma = \frac{1}{2}$ and let $R \rightarrow \infty$, then we get $|A|^2 = 0$ everywhere on $\Gamma$, so $\Gamma$ is a hyperplane.

  \section{Partial regularity for tangent flows}Suppose $\{\Gamma_s\}$ is a tangent flow of $\{M_t\}$ at the first singular time, and $\{M^j_s\}$ is the corresponding sequence of parabolic rescalings of $\{M_t\}$. We will need the following consequence of Theorem \ref{Ecker}:

  \begin{lemm}\label{mainlemma}
    Let $\{M_t\} \subset \emph{R}^{n+1}$ be  closed hypersurfaces flowing by mean curvature flow, and $\{\Gamma_s\}$ and $\{M^j_s\}$ are defined as above. If $\Gamma_{-1}$ is multiplicity one, then for any compact subset $K \subset Reg(\Gamma_{-1})$ there is a subsequence of $\{M^j_{-1}\}$ which converge smoothly to $\{\Gamma_{-1}\}$ on $K$.
  \end{lemm}

  Before proving the lemma, we need the following result:

  \begin{lemm}\label{lemm}
     Suppose $\{\Gamma_s\}$ and $\{M_s^j\}$ are defined as in the above lemma, $\Gamma_{-1}$ is multiplicity one, and $\epsilon > 0$ is any fixed positive constant. Let $Reg(\Gamma_{-1})$ represent the regular part of $\Gamma_{-1}$. Then for any $x_0 \in Reg(\Gamma_{-1})$, there exist  $\rho_0 = \rho_0(x_0) > 0$ and some $\rho \in (0,\rho_0)$ and a sufficiently large $J$, such that
     $$\int_{M^j_s} \Phi_{(y,\tau)} \phi_{(y,\tau),\rho_0} \leq 1 + \epsilon$$
     for all $(y,\tau) \in B_{\rho}(x_0) \times (-1 - \rho^2, -1)$, $s \in (\tau - \rho^2, \tau)$ and $j > J$.
  \end{lemm}
  \begin{proof}
    Because $x_0$ is a regular point of $\Gamma_{-1}$ and $\{\Gamma_s\}$ is a self-shrinking mean curvature flow, we can find a $\rho_0 = \rho_0(x_0) > 0$ such that $\{\Gamma_s\}$ is smooth on $B_{\sqrt{1+2n}\rho_0}(x_0) \times (-1 - \rho_0 , \-1)$.

    Since $\Gamma_{-1}$ is multiplicity one, so $\Theta(\Gamma_s, x_0, -1) = 1$. By Theorem {\ref{locmono}}, we can find a $\rho_1 \in (0,\rho_0]$ such that
    $$\int_{\Gamma_{-1 - \rho^2_{1}}} \Phi_{(x_0 , -1)} \phi_{(x_0, -1),\rho_0} \leq 1+ \frac{1}{4} \epsilon.$$
    The continuity of
    $$(y,\tau) \longrightarrow \int_{\Gamma_{-1 - \rho_1^2}} \Phi_{(y,\tau)}\phi_{(y,\tau),\rho_0}$$
    implies that for some $\rho \in (0,\rho_0)$ and all $(y,\tau) \in B_{\rho}(x_0) \times (-1 -\rho^2, -1)$,
    \begin{equation}\label{gamma}
    \int_{\Gamma_{-1 - \rho_1^2}}\Phi_{(y, \tau)}\phi_{(y,\tau),\rho_0} \leq 1 + \frac{1}{2}\epsilon
    \end{equation}
    and furthermore $(\tau - \rho^2, \tau) \subset (-1 - \rho_1^2 , -1)$.
    Define a sequence of functions $g_j$ by
    $$g_j(y,\tau) = \int_{M^j_{-1 - \rho_1^2}}\Phi_{(y, \tau)}\phi_{(y,\tau),\rho_0}$$
    We will only consider the $g_j's$ on the region $\overline{B}_{\rho}(x_0) \times [-1 -\rho^2, -1]$, it follows from the first variation formula, see lemma 3.7 in \cite{CM}, that $g_j$'s are uniformly Lipschitz in this region with
    $$\sup_{\overline{B}_{\rho}(x_0) \times [-1 -\rho^2, -1]} |\nabla_{y,\tau} g_j| < C,$$
    where $C$ depends on $\rho$ and the scale-invariant local area bounds for the $M^j_{-1 - \rho_1^2}$'s which are uniformly bounded. Since $M^j_{-1 - \rho_1^2}$'s converge to $\Gamma_{-1 - \rho_1^2}$ as Radon measures and $\Gamma_{-1 - \rho_1^2}$ satisfies (\ref{gamma}), so there exists some $J$ sufficiently large so that for all $j > J$ we have
   $$\int_{M^j_{-1 - \rho_1^2}} \Phi_{(y,\tau)} \phi_{(y,\tau),\rho_0} \leq 1 + \epsilon$$
     for all $(y,\tau) \in B_{\rho}(x_0) \times (-1 - \rho^2, -1)$. Since by lemma \ref{locmono},
     $$s \mapsto \int_{M_s^j} \Phi_{(y,\tau)} \phi_{(y,\tau),\rho_0} $$
     is non-increasing we obtain
     $$\int_{M_s^j} \Phi_{(y,\tau)} \phi_{(y,\tau),\rho_0} \leq \int_{M^j_{-1 - \rho_1^2}} \Phi_{(y,\tau)} \phi_{(y,\tau),\rho_0} \leq 1 + \epsilon$$
     for all  $(y,\tau) \in B_{\rho}(x_0) \times (-1 - \rho^2, -1)$, $s \in (\tau - \rho^2, \tau)$ and $j > J$.
  \end{proof}

  \textit{Proof of lemma \ref{mainlemma}}: By Lemma \ref{lemm} and Theorem \ref{Ecker}, we know that for any $x_0 \in Reg(\Gamma_{-1})$, there is a positive $\rho(x_0)$ and a sufficiently large number $J = J(x_0)$ such that $\{M^j_{-1}\}$ have unform bound on second fundamental form. From the curvature estimate of mean curvature flow, we can also get unform bound on higher derivatives of second fundamental form of $\{M^j_{-1}\}$. Therefore, for any compact subset $K \subset Reg(\Gamma_{-1})$ we can choose a subsequence of $\{M^j_{-1}\}$ denoted by $\{M^{j_i}_{-1}\}$, such that $\{M^{j_i}_{-1}\}$ converge smoothly to $\{\Gamma_{-1}\}$ on $K$.

\subsection{Proof of Theorem \ref{main2}}We will prove Theorem \ref{main2} by mean curvature flows. Suppose $M_0$ is a hypersurface in $\textbf{R}^{n+1}$ satisfying all conditions in Theorem \ref{main2}, and denote $\{M_t\}$ is the mean curvature flow starting from $M_0$ before the first singular time. Since mean curvature $H > 0$ on $M_0$, by Theorem 4.3 in \cite{H}, we obviously have the following lemma:
\begin{lemm}\label{pinchlemma}
  Suppose $\{M_t\}_{t \in [0,T)}$ is a mean curvature flow  before the first singular time starting from $M_0$. If there is a constant $C$ such that $|A|^2 \leq CH^2$ on $M_0$, then we have
  \begin{equation}\label{pinch}
    |A|^2(x,t) \leq CH^2(x,t)
  \end{equation}
  holds on $M_t$ for every $t \in [0,T)$.
\end{lemm}
For preparation of proving Theorem \ref{main2}, we also need the following two important theorems.
\begin{theo}[see \cite{CM}] \label{CM1}
  $\textbf{S}^{k} \times \textbf{R}^{n-k}$ are the only smooth complete embedded self-shrinkers without boundary, with polynoimal volume growth, and $H \geq 0$ in $\textbf{R}^{n+1}$
\end{theo}

\begin{theo}[see \cite{CIMW}] \label{CM2}
  If $\Gamma \subset \textbf{R}^{n+1}$ is a weak solution of the self-shrinker equation (\ref{selfshrinker}), $\lambda(\Gamma) < \frac{3}{2}$, and there is a constant $C >0$ such that
  $$|A| \leq CH $$
  on the regular set $ Reg(\Gamma)$, then $\Gamma$ is smooth.
\end{theo}

\textit{Proof of Theorem \ref{main2}}: Assume $\{M_t\}_{t \in [0,T)}$ is a mean curvature flow starting from $M_0$, $T$ is the first singular time and $x_0$ is a singular point in $\textbf{R}^{n+1}$. By Lemma \ref{tang.exist} and Lemma \ref{tan.conv}, we know that there exist a tangent flow $\{\Gamma_s\}_{s<0}$ at $(x_0 , T)$ and a corresponding sequence $\{M^j_s\}$ of parabolic transformations of $\{M_t\}$.
By Lemma \ref{pinchlemma} and inequality (\ref{pinch}) is scaling-invariant, we get that for every $j$,
$$|A_j| \leq  C H_j$$
on $M^j_s$, where $A_j$ and $H_j$ are the second fundamental form and mean curvature on $M^j_s$ respectively. Combining this with Lemma \ref{mainlemma}, we have
$|A| \leq C H $
on the regular part of $\Gamma_{-1}$. By Theorem \ref{CM1} and Theorem \ref{CM2}, we know that
$\Gamma_{-1}$ must be of the form $\textbf{S}^{k} \times \textbf{R}^{n-k}$. Since entropy is non-increasing under mean curvature flow, scaling non-invariant and lower semi-continuous under limits, then we have
$$\lambda(\Gamma_{-1}) \leq \min\{\lambda(\textbf{S}^{n-1}), \frac{3}{2}\}$$
If $\lambda(\Gamma_{-1}) = \min\{\lambda(\textbf{S}^{n-1}), \frac{3}{2}\}$, we know that the entropy $\lambda(M_t)$ is invariant under $\{M_t\}$, By Huisken's monotonicity formula, $M_0$ must be a compact self-shrinker  with $H > 0$, then from Theorem \ref{CM1} we know that $M_0$ must be a round sphere.

If $\lambda(\Gamma_{-1}) < \min\{\lambda(\textbf{S}^{n-1}), \frac{3}{2}\}$, By Theorem \ref{CM1} and Theorem \ref{CM2}, we know that $\Gamma_{-1}$ must be $\textbf{S}^{n}$.

Using Lemma \ref{mainlemma} again, we have for sufficient large $j$, $M^j_{-1}$ can be written as a smooth graph over $\Gamma_{-1}$. Since $\Gamma_{-1}$ is a round sphere, then we have for sufficient large $j$, we have $M^j_{-1}$ is diffeomorphic to a round sphere. By the definition of $M^j_{-1}$, we know that $M^j_{-1} = \lambda_j^{-1}(M_{T - \lambda_j^2} - x_0)$. Since mean curvature flow $\{M_t\}$ is smooth up to the first singular time, then $M_0$ is diffeomorphic to a round sphere, and we complete the proof of Theorem \ref{main2}.

\begin{rema}
We think entropy may give some information of the hypersurface, so this work is attempt to study the singularities of mean curvature flow by entropy. From theorem \ref{main2} we can see that if the entropy of a mean convex compact hypersurface is no more than $\min\{\lambda(\textbf{S}^{n-1}), \frac{3}{2}\}$ , then it is diffeomorphic to a round sphere. We believe that if the entropy is a little higher we can also get some classification result as we mentioned at the beginning of this paper.
\end{rema}


\begin{thebibliography}{99}
\bibitem{E}
K.Ecker, Regularity theory for mean curvature flow. BirkH\"{a}user, Boston, 2004.
\bibitem{EL}
J.Ells, L.Lemaire, A report on harmonic maps. Bull. London Math. Soc., 10(1978), 1-68.
\bibitem{CM}
T.H.Colding and W.P.Minicozzi II, Generic mean curvature flow I; generic singularities. Annals of Math., 175(2012), 755-833.
\bibitem{CIMW}
T.H.Colding, T.Ilmanen, W.P.Minicozzi II and B.White, The round sphere minimizes entropy among closed self-shrinkers. J. Differential Geom. Volume 95, Number 1 (2013), 53-69.
\bibitem{H}
G.Huisken, Flow by mean curvature of convex surfaces into spheres. J. Differential Geom., 20(1984), 237-266.
\bibitem{HS}
G.Huisken and C.Sinestrari, Mean curvature flow with surgeries of two-convex hypersurfaces. Invent. Math., 175(2009), 137-221.
\bibitem{I1}
T.Ilmanen, Singularities of mean curvature flow of surfaces. preprint, 1995.
\bibitem{I2}
T.Ilmanen, Elliptic regularization and partial regularity for motion by mean curvature. preprint, 1993.
\bibitem{L}
Liao,G.J, A regularity theorem for harmonic map with small energy. J. Differential Geometry, 22(1985), 233-241.
\bibitem{mS}
M.Struwe, Uniqueness of harmonic maps with small energy. Manuscripta Mathematica, 96(1998), 463-486.
\bibitem{S}
A.Stone, A density function and the structure of singularities of the mean curvature flow. Calc.
Var. 2(1994), 443-480.
\bibitem{SW}
J.Sacks, K.Uhlenbeck, The existence of minimal immersions of 2-spheres. Ann. of Math., (2)113(1981), 1-24.
\bibitem{W}
B.White, A local regularity theorem for mean curvature flow. Ann. of Math. (2)161(2005), 1487-1519.
 \end{thebibliography}
 \end{document}